\documentclass[10pt]{amsart}
\usepackage{amssymb,latexsym,amsmath}
\usepackage{graphics}                 
\usepackage{color}                    
\usepackage{hyperref}                 
\usepackage[all]{xy}

\begin{document}

\newtheorem{definition}{Definition}
\newtheorem{theorem}{Theorem}
\newtheorem{proposition}[theorem]{Proposition}
\newtheorem{lemma}[theorem]{Lemma}
\newtheorem{corollary}[theorem]{Corollary}
\newtheorem{question}[theorem]{Question}
\newtheorem{remark}[theorem]{Remark}
\newtheorem{example}[theorem]{Example}
\newtheorem{conjecture}[theorem]{Conjecture}

\newtheorem{correction}{Correction}

\def\A{{\mathbb{A}}}
\def\C{{\mathbb{C}}}

\def\L{{\mathbb{L}}}
\def\O{{\mathcal{O}}}
\def\P{{\mathbb{P}}}
\def\Q{{\mathbb{Q}}}
\def\Z{{\mathbb{Z}}}
\def\Ch{{\rm Ch}}
\def\p{{\mathbf{p}}}

\def\sp{{\rm SP}}
\def\rank{{\rm rank}}

\def\mC{{\mathcal{C}}}
\def\D{{\mathcal{D}}}

\title{Notes on  Algebraic Cycles and Homotopy theory}
\author{Wenchuan Hu}
\date{November 26, 2008}

\keywords{Algebraic cycles; Chow variety; $d$-connectedness}
\address{Department of Mathematics, Massachusetts Institute of Technology, Room 2-363B,
 77 Massachusetts Avenue,
 Cambridge, MA 02139, USA
}

\email{wenchuan@math.mit.edu}

\begin{abstract}
We show  that a conjecture by Lawson holds, that is, the inclusion
from the Chow variety $C_{p,d}(\P^n)$ of all effective algebraic $p$-cycles of
degree $d$ in $n$-dimensional projective space $\P^n$ to the space $\mC_{p}(\P^n)$
of effective algebraic $p$-cycles in $\P^n$ is $2d$-connected. As a result,
 the homotopy and homology groups of $C_{p,d}(\P^n)$ are calculated up to  $2d$.
We also show an analogous statement for Chow variety $C_{p,d}(\P^n)$
over algebraically closed field $K$ of arbitrary characteristic
and compute their etale homotopy groups up to $2d$.
\end{abstract}

\maketitle

\tableofcontents

\section{Introduction}
Let $\P^n$ be the  complex projective space of  dimension $n$ and let $C_{p,d}(\P^n)$
be the space of effective algebraic $p$-cycles of degree $d$ on $\P^n$. A
fact proved by Chow and Van der Waerden is that $C_{p,d}(\P^n)$ carries the structure of a closed complex algebraic set.
Hence it carries the structure of a compact Hausdorff space.

Let $l_0\subset \P^n$ be a fixed $p$-dimensional linear subspace. For each $d\geq 1$, we consider the analytic embedding
\begin{equation}\label{eqn1}
i:C_{p,d}(\P^n)\hookrightarrow C_{p,d+1}(\P^n)
\end{equation}
defined by $c\mapsto c+l_0$. From this sequence of embeddings we can form the union
$$
\mC_{p}(\P^n)=\lim_{d\to\infty} C_{p,d}(\P^n).
$$

The topology on $\mC_{p}(\P^n)$ is the weak topology for $\{ C_{p,d}(\P^n)\}_{d=1}^{\infty}$, that is, a set $C\subset \mC_{p}(\P^n)$
is closed if and only if $C\cap  C_{p,d}(\P^n)$ is closed in $C_{p,d}(\P^n)$ for all $d\geq 1$. For details, the reader is referred to the
paper \cite{Lawson1}. The background on homotopy theory is referred to the book \cite{Whitehead}.

In this note, we will prove the following main result.
\begin{theorem}\label{Th2}
For all $n,p$ and $d$, the  inclusion  $i:C_{p,d}(\P^n)\hookrightarrow \mC_{p}(\P^n)$ induced by Equation (\ref{eqn1})
is $2d$-connected.
\end{theorem}

This fact was conjectured by Lawson \cite{Lawson1}, who had proved that
the inclusion $i:C_{p,d}(\P^n)\hookrightarrow \mC_{p}(\P^n)$
has a right homotopy inverse through dimension $2d$.

The method in the proof of Theorem \ref{Th2}  comes from Lawson in his proof of the
Complex Suspension Theorem \cite{Lawson1}, i.e., the complex suspension to the space of $p$-cycles on $\P^n$
yields a homotopy equivalence to the space of $(p+1)$-cycles on  $\P^{n+1}$.

As applications of the main result, we calculate the first $2d+1$ homotopy  and homology groups of $C_{p,d}(\P^n)$.

The analogous result of  Theorem \ref{Th2} for Chow varieties over algebraically closed fields is obtained (cf. Theorem \ref{Th4.1}).

\thanks{\emph{Acknowledgements}: I would like to gratitude professor Blaine Lawson for pointing out a gap in an early version of this note
as well as helpful advice on the organization of the paper.}

\section{The Method in the proof of the Complex Suspension Theorem}

Now we briefly review  Lawson's method in the proof of the Complex Suspension Theorem. The materials in this section
can be found in \cite{Lawson1}, \cite{Friedlander1} and \cite{Lawson2}.

Fix a hyperplane $\P^n\subset\P^{n+1}$ and a point $\P^0\in \P^{n+1}-\P^{n}$.
For any non-negative integer $p$ and $d$, set
$$T_{p+1,d}(\P^{n+1}):=\big\{c=\sum n_iV_i\in C_{p+1,d}(\P^{n+1})|\dim(V_i\cap\P^n)=p, ~\forall i\big \}.$$
(when $d=0$, $C_{p,0}(\P^n)$ is defined to be the empty cycle.)

The following result was proved by Lawson and a generalized algebraic version was proved by Friedlander \cite{Friedlander1}.
 \begin{proposition}[\cite{Lawson1}]\label{prop1.2}
The set  $T_{p+1,d}(\P^{n+1})$  is Zariski open in $C_{p+1,d}(\P^{n+1})$. Moreover, $T_{p+1,d}(\P^{n+1})$ is homotopy equivalent to
$C_{p,d}(\P^{n})$. In particular, their corresponding homotopy groups are isomorphic, i.e.,
\begin{equation}\label{eq12}
\pi_*(T_{p+1,d}(\P^{n+1}))\cong\pi_*(C_{p,d}(\P^{n})).
\end{equation}
\end{proposition}

Fix linear embedding $\P^{n+1}\subset \P^{n+2}$ and two points $x_0,x_1\in \P^{n+2}-\P^{n+1}$.
Each projection $p_i:\P^{n+2}-\{x_0\}\to \P^{n+1}$ ($i=0,1$) gives us a holomorphic line bundle over $\P^{n+1}$.

Let $D\in C_{n+1,e}(\P^{n+2})$ be an effective divisor of degree $e$ in $\P^{n+2}$
such that $x_0, x_1$ are not in $D$. Denote by $\widetilde{Div_e}(\P^{n+2})\subset C_{n+1,e}(\P^{n+2})$ the subset of all
such $D$.

Any effective cycle $c\in C_{p+1,d}(\P^{n+1})$
 can be lifted to a cycle with support in $D$, defined as follows:
$$ \Psi_D(c)=(\Sigma_{x_0}c)\cdot D.
$$

 The map $\Psi(c, D):=\Psi_D(c)$ is a continuous map with variables $c$ and $D$.
Hence we have a continuous map $\Psi_D:C_{p+1,d}(\P^{n+1})\to C_{p+1,de}(\P^{n+2}-\{x_0,x_1\})$.
The composition of $\Psi_D$ with the projection $(p_0)_*$ is
$(p_0)_*\circ\Psi_D=e$ (where $e\cdot c=c+\cdots+c$ for $e$ times).
The composition of $\Psi_D$ with the projection $(p_1)_*$ gives us a transformation of cycles in
 $\P^{n+1}$ which makes most of them intersecting properly to $\P^{n}$.
To see this, we consider the family of divisors $tD$, $0\leq t\leq 1$, given by scalar multiplication by $t$ in the line bundle
$p_0:\P^{n+2}-\{x_0\}\to \P^{n+1}$.

Assume $x_1$ is not in $tD$ for all $t$. Then the above construction gives us a family transformation
$$ F_{tD}:=(p_1)_*\circ \Psi_{tD}: C_{p+1,d}(\P^{n+1})\to C_{p+1,de}(\P^{n+1})
$$
for $0\leq t\leq 1$. Note that $F_{0D}\equiv e$ (multiplication by $e$).

The question is that for a fixed $c$, which divisors $D\in C_{n+1,e}(\P^{n+2})$
($x_0$ is not in $D$ and $x_1$ is not in $\bigcup_{0\leq t\leq 1} tD$) have the property that
$$ F_{tD}(c)\in T_{p+1,de}(\P^{n+1})
$$
for all $0< t\leq 1$.

Set $B_c:=\{D\in C_{n+1,e}(\P^{n+2})|F_{tD}(c) ~\hbox{is not in $T_{p+1,de} (\P^{n+1})$ for some  $0<t\leq 1$}\}$,
i.e., all degree $e$ divisors on $\P^{n+2}$
 such that some component of
 \begin{equation*}
 (p_1)_*\circ \Psi_{tD}(c)\subset \P^n
 \end{equation*}
 for some $t>0$.

An important calculation we will use later is the following result.
\begin{proposition}[\cite{Lawson1}]
For $c\in C_{p+1,d}(\P^{n+1})$, ${\rm codim}_{\C}B_c\geq \big(^{p+e+1}_{\quad e}\big)$.
\end{proposition}

\section{Proof of the first main result}
In the construction of the last section, $F_{tD}$ maps $C_{p+1,d}(\P^{n+1})$ to $C_{p+1,de}(\P^{n+1})$, i.e.,
$$ F_{tD}:=(p_1)_*\circ \Psi_{tD}: C_{p+1,d}(\P^{n+1})\to C_{p+1,de}(\P^{n+1}).
$$
Moreover, the image of $F_{tD}$ is in the Zariski open subset $T_{p+1,d}(\P^{n+1})$ if $D$ is not $B_c$.
We can find such a $D$ if ${\rm codim}_{\C}B_c\geq \big(^{p+e+1}_{\quad e}\big)$ is positive.

Suppose now that $f:S^k\to C_{p+1,d}(\P^{n+1})$ is a continuous map for $0<k\leq 2d$.
We may assume that $f$ is piecewise linear up to homotopy. Then the map $e\cdot f=F_{0D}\circ f$ is homotopic
to a map $S^k\to T_{p+1,de}(\P^{n+1})$. To see this, we consider the family
$$F_{tD}\circ f:S^k\to C_{p+1,de}(\P^{n+1}), \quad 0\leq t\leq 1,
$$
where $D$ lies outside the union $\bigcup_{x\in S^k}B_{f(x)}$. This is a set of real
codimension bigger than or equal to $2\big(^{p+e+1}_{\quad e}\big)-(k+1)$.
Therefore, if $2\big(^{p+e+1}_{\quad e}\big)-(k+1)\geq 1$, i.e.,
$k\leq 2\big(^{p+e+1}_{\quad e}\big)-2$, then such a $D$ exists once
 we choose a large $e$ such that $2d\leq 2\big(^{p+e+1}_{\quad e}\big)-2$.
Therefore we have the following commutative diagram
\begin{equation}\label{eq2}
\xymatrix{
&T_{p+1,d}(\P^{n+1})\ar[r]\ar@{^(->}@{^(->}[d]&T_{p+1,de}(\P^{n+1})\ar@{^(->}[d]\\
S^k\ar[r]^-{f}&C_{p+1,d}(\P^{n+1})\ar[r]^-{F_{tD}}\ar[ur]^{F_D}& C_{p+1,de}(\P^{n+1}),
}
\end{equation}
where $F_D:=F_{1D}$.

\begin{proposition} \label{lemma1.4}
For any integer $e\geq 1$, the  map
$$\phi_e: C_{0,d}(\P^{n})\to C_{0,de}(\P^{n}), ~ \phi_e(c)=e\cdot c$$
 induces injections
$\phi_{e*}: \pi_k(C_{0,d}(\P^{n}))\to \pi_k(C_{0,de}(\P^{n}))$ for $k\leq 2d$.
\end{proposition}

\begin{proof}
First note that $C_{0,d}(\P^n)\cong \sp^d(\P^n)$, where $\sp^d(\P^n)$
denotes the $d$-th symmetric product of $\P^n$.
Denote by $\Delta:(\P^n)^d\to ((\P^n)^d)^e=(\P^n)^{de}$ the diagonal map
$\Delta(x)=(x,...,x)$ for  $e$ copies of $x$ and  and $p_1:((\P^n)^d)^e\to (\P^n)^{d}$
the projection on the first component. Hence we have $p_1\circ \Delta= id:(\P^n)^d\to (\P^n)^d$
and $p_{1*}\circ \Delta_*= id_*:H_k((\P^n)^d)\to H_k((\P^n)^d)$ for any integer $k\geq 0$.
This implies the injectivity of $\Delta_*$.

From the commutative diagram of continuous maps of complex varieties
$$\xymatrix{
(\P^n)^d\ar[r]^{\Delta}\ar[d]^-{\pi} &(\P^n)^{de}\ar[d]^-{\pi}\\
\sp^d(\P^n)\ar[r]^-{ \phi_e}&\sp^{de}(\P^n),\\
}
$$
where $\pi:X^m\to \sp^m X $ is the natural projection,
we have the induced commutative diagram on homology groups
\begin{equation}\label{eqn4}
\xymatrix{
H_k((\P^n)^d,\Q)\ar[r]^{\Delta_*}\ar[d]^-{\pi_*} &H_k((\P^n)^{de},\Q)\ar[d]^-{\pi_*}\\
H_k(\sp^d(\P^n),\Q)\ar[r]^-{\phi_{e*}'\otimes \Q}& H_k(\sp^{de}(\P^n),\Q),\\
}
\end{equation}
for any $k\geq 0$.

Now we show that  $\phi_{e*}'\otimes \Q$ is  injective for all $e\geq 1$.
 Let $\alpha\in H_k(\sp^d(\P^n),\Q)$ be an element such that
 $\phi_{e*}'\otimes \Q(\alpha)=0$. Since $H_k((\P^n)^d,\Q)^{S_d}\cong H_k(\sp^d(\P^n),\Q)$ for any $d\geq 1$,
 there is an element $\tilde{\alpha}\in H_k((\P^n)^d,\Q)$ such that $\pi_*(\tilde{\alpha})=\alpha$,
 where $S_d$ is the $d$-th symmetric group
 and $H_k((\P^n)^d,\Q)^{S_d}$ is the $S_d$-invariant subgroup of $H_k((\P^n)^d,\Q)$.
 The element $\tilde{\alpha}$ is $S_d$-invariant. Set $\tilde{\beta}:=\Delta_*(\tilde{\alpha})$.
 From the commutative diagram (\ref{eqn4}), we have $\pi_*(\tilde{\beta})=0$.
  Since $\tilde{\beta}$ is $S_{de}$-invariant and  $\pi_*(\tilde{\beta})=0$,
  we get $\tilde{\beta}=0$ since $H_k((\P^n)^{de},\Q)^{S_{de}}\cong H_k(\sp^{de}(\P^n),\Q)$ (cf. e.g. \cite{Elizondo-Srinivas}).
   This implies that $\tilde{\alpha}=0$ since $\Delta_*$ is injective and
   $\Delta_*(\tilde{\alpha})=\tilde{\beta}=0$. Therefore $\alpha=\pi_*(\tilde{\alpha})=0$,
   i.e.,  $\phi_{e*}'\otimes \Q$ is injective on rational homology groups.

Note that the map $\phi_e:\sp^d(\P^n)\to\sp^{de}(\P^n)$ induces  a commutative diagram
\begin{equation}\label{eqn5}
\xymatrix{\pi_k(\sp^d(\P^n))\ar[r]^{\phi_{e*}}\ar[d]^-{\rho} &\pi_k(\sp^{de}(\P^n))\ar[d]^-{\rho}\\
H_k(\sp^d(\P^n))\ar[r]^-{\phi_{e*}'}& H_k(\sp^{de}(\P^n)),
}
\end{equation}
where $\rho$ are  Hurewicz maps.

 We claim that if $k\leq 2d$,
then $\phi_{e*}:\pi_k(\sp^d(\P^n))\to\pi_k(\sp^{de}(\P^n))$ is injective.
First we will show that $\phi_{e*}\otimes \Q:\pi_k(\sp^d(\P^n))\otimes \Q\to\pi_k(\sp^{de}(\P^n))\otimes \Q$ is injective.

Fix $x_0\in \P^n$. Let  $i:\P^n\to \sp^d(\P^n)$ be the map given by
$i(x)=x+(d-1)x_0$ and let $j:\sp^d(\P^n)\to \sp^{\infty}(\P^n)$ be induced by
the sequence of maps  $\sp^d(\P^n)\to \sp^{d+m}(\P^n), y\mapsto y+m x_0$.
Then $i$ induces a commutative diagram
\begin{equation}\label{eqn6}
\xymatrix{\pi_k(\P^n)\ar[r]^{i_{*}}\ar[d]_-{\rho} &\pi_k(\sp^{d}(\P^n))\ar[r]^{j_*}\ar[d]_-{\rho}&\pi_k(\sp^{\infty}(\P^n))\ar @{.>}[dll]^{DT}\\
H_k(\P^n)\ar[r]^-{i_*'}& H_k(\sp^{d}(\P^n)),&
}
\end{equation}
where the dotted map $DT:\pi_k(\sp^{d}(\P^n))\dashrightarrow H_k(\P^n)$ is
the Dold-Thom isomorphism. This follows from the fact that, for any connected finite CW complex
$X$ and a fixed point $x_0\in X$, the composed map
$\pi_k(X)\stackrel{j_*}{\to} \pi_k(\sp^{\infty}(X))\stackrel{DT}{\longrightarrow} {H}_k(X)$
of the induced  map $j_*$ by the inclusion $j:X=\sp^1(X)\subset \sp^{\infty}(X)$ and the Dold-Thom map $DT:
\pi_k(\sp^{\infty}(X))\to H_k(X)$ is the Hurewicz map (cf. \cite{Dold-Thom}).

Note that $j_*:\pi_k(\sp^d(\P^n))\to \pi_k(\sp^{\infty}(\P^n))$
is an isomorphism for $k\leq 2d$ (cf. \cite{Dold} or \cite{Milgram}). From equation (\ref{eqn6}), we obtain the injectivity
of the Hurewicz map $\rho: \pi_k(\sp^{d}(\P^n))\to H_k(\sp^{d}(\P^n))$ for $k\leq 2d$
since $i_*'\circ DT\circ j_* = \rho$ and the injectivity of $i_*'$.
Now the injectivity of
$\phi_{e*}\otimes \Q:\pi_k(\sp^d(\P^n))\otimes \Q\to \pi_k(\sp^{de}(\P^n))\otimes \Q$ follows from
equation (\ref{eqn5}) as well as the injectivity of $\rho\otimes\Q$ and $\phi_{e*}'\otimes \Q$.
Since  $$\pi_k(\sp^d(\P^n))\cong \pi_k(\sp^{de}(\P^n))\cong\left\{
\begin{array}{lll}
\Z,& \hbox{$0<k\leq 2d$ and $k$ even,}\\
0,& \hbox{$k=0$ or $k\leq 2d$ and $k$ odd}
\end{array}
\right.$$
(cf. \cite{Dold} or \cite{Milgram}), the injectivity of $\phi_{e*}\otimes \Q$ implies the  injectivity of $\phi_{e*}$.
\end{proof}

\begin{remark}\label{remark1.5}
From the proof of Proposition \ref{lemma1.4}, we obtain that
 $\phi_{e*}\otimes \Q:\pi_k(C_{0,d}(\P^n))\otimes\Q\stackrel{\cong}{\to} \pi_k(C_{0,de}(\P^n))\otimes \Q$
  for all $k\leq 2d$. To see this we note that, for $k\leq 2d$,
 both $\pi_k(C_{0,d}(\P^n))$ and $\pi_k(C_{0,de}(\P^n))$ are isomorphic to
 $\pi_k(C_{0}(\P^n))\cong H_k(\P^n)$. So the injectivity of $\phi_{e*}$ implies
 an isomorphism for $\phi_{e*}\otimes \Q$.
\end{remark}

\begin{lemma} \label{lemma1.6}
There is a commutative diagram
$$
\xymatrix{C_{p,d}(\P^{n})\ar[r]^{\Phi_{p,d,n,e}}\ar[d]^{\Sigma}& C_{p,de}(\P^{n})\ar[d]^{\Sigma}\\
T_{p+1,d}(\P^{n+1})\ar[r]^{F_{tD}}& T_{p+1,de}(\P^{n+1}),
}
$$
where $\Phi_{p,d,n,e}$ ($\Phi_{0,d,n,e}=\phi_e$ in Proposition \ref{lemma1.4}) is the composed map
$$
\begin{array}{ccccc}
 C_{p,d}(\P^{n})&\to& C_{p,d}(\P^{n})\times\cdots \times C_{p,d}(\P^{n}) &\to& C_{p,de}(\P^{n})\\
 c&\mapsto& (c,...,c)&\mapsto& e\cdot c
\end{array}
$$
and $D\in \widetilde{Div_e}(\P^{n+2})$, $F_{tD}$ is the restriction of $F_{tD}: C_{p+1,d}(\P^{n+1})\to C_{p+1,de}(\P^{n+1})$.
\end{lemma}
\begin{proof}

Note that the image of the restriction of $$F_{tD}: C_{p+1,d}(\P^{n+1})\to C_{p+1,de}(\P^{n+1})$$ on $T_{p+1,d}(\P^{n+1})$ is
in $T_{p+1,de}(\P^{n+1})$. The remaining part follows from Lemma 5.5 in \cite{Lawson1}.
\end{proof}

\begin{proposition}\label{prop1.7}
For integers $p,d,n$, there is an integer $e_{p,d,n}\geq 1$ such that if $e\geq e_{p,d,n}$,
then the map $\Phi_{p+1,d,n+1,e}:C_{p+1,d}(\P^{n})\to C_{p+1,de}(\P^{n+1})$ given by
$c\mapsto e\cdot c$ induces injections
$$(\Phi_{p+1,d,n+1,e})_*: \pi_k(C_{p+1,d}(\P^{n+1}))\to \pi_k(C_{p+1,de}(\P^{n+1}))$$ for $k\leq 2d$.
\end{proposition}

\begin{proof}
 We prove it by induction. The case that $p=-1$ follows from Proposition \ref{lemma1.4}. We assume that
$\Phi_{p,d,n,e}: C_{p,d}(\P^{n})\to C_{p,de}(\P^{n+1})$ defined by $\Phi_{p,d,n,e}(c)=e\cdot c$ induces injections
$(\Phi_{p,d,n,e})_*: \pi_k(C_{p,d}(\P^{n}))\to \pi_k(C_{p,de}(\P^{n}))$ for $k\leq 2d$ and $e\geq e_{p,d,n}$.

Let $\alpha\in \pi_k(C_{p+1,d}(\P^{n+1}))$ be an element such that  $(\Phi_{p+1,d,n+1,e})_*(\alpha)=0$, that is, $(F_{0D})_*(\alpha)=0$.
Let $f:S^k\to C_{p+1,d}(\P^{n+1})$ be piecewise linear up to homotopy such that $[f]=\alpha$.
By assumption, $[F_{0D}\circ f]=0$.

By Lemma \ref{lemma1.6} and Equation (\ref{eq2}), we have
$$
\xymatrix{S^k\ar[dd]^{||}\ar[r]^-{g}&C_{p,d}(\P^{n})\ar[r]^{\Phi_{p,d,n,e}}\ar[d]^{\Sigma}& C_{p,de}(\P^{n})\ar[d]^{\Sigma}\\
&T_{p+1,d}(\P^{n+1})\ar[r]\ar@{^(->}@{^(->}[d]& T_{p+1,de}(\P^{n+1})\ar@{^(->}[d]\\
S^k\ar[r]^-{f}&C_{p+1,d}(\P^{n+1})\ar[r]^-{F_{tD}}\ar[ur]^{F_D}& C_{p+1,de}(\P^{n+1}).
}
$$
Since $F_{0D}$ is homotopy to $F_{D}:C_{p+1,d}(\P^{n+1}))\to T_{p+1,de}(\P^{n+1}))$
for $e\geq e_{p+1,d,n}$ (cf. \cite{Lawson1}), we have
$[F_D\circ f]=0\in \pi_k(T_{p+1,de}(\P^{n+1}))$.
By Proposition \ref{prop1.2}, $\Sigma^{-1}_*([F_D\circ f])=0$.
From the above commutative diagram and the injectivity of   $(\Phi_{p,d,n,e})_*$,
the map $f:S^k\to C_{p+1,d}(\P^{n+1})$ can be lifted to
a null homotopy map $g:S^k\to C_{p,d}(\P^{n})$
such that $(\Phi_{p,d,n,e})_*([g])=\Sigma^{-1}_*([F_D\circ f])$
and $[f]=[\Sigma\circ g]$. Hence $\alpha=[f]=0$. That is,
$(F_{0D})_*=(\Phi_{p+1,d,n+1,e})_*$ is injective  for $k\leq 2d$.
\end{proof}

\begin{proof}[The proof of Theorem \ref{Th2}]
 The case that $p=-1$ has been proved in \cite{Dold} and \cite{Milgram}.
By taking limit $e\to \infty$ in Proposition \ref{prop1.7},
we get injections $\pi_k(C_{p+1,d}(\P^{n+1}))\to \pi_k(\mC_{p+1}(\P^{n+1}))$ for $k\leq 2d$.
 On one hand, from the fact that $\pi_k(\mC_{p+1}(\P^{n+1}))$ is isomorphic to
 either $\Z$ or $0$ (cf. \cite{Lawson1}) and the injections above,
 we obtain $\pi_k(C_{p+1,d}(\P^{n+1}))$ is isomorphic to either  $\Z$ or $0$ for $k\leq 2d$.
 On the other hand, the induced map $i_*:\pi_k(C_{p+1,d}(\P^{n+1}))\to \pi_k(\mC_{p+1}(\P^{n+1}))$
 by equation (\ref{eqn1}) is surjective for $k\leq 2d$ (cf. \cite{Lawson1}, Theorem 2). Hence
 $\pi_k(C_{p+1,d}(\P^{n+1}))$  is isomorphic to $\pi_k(\mC_{p+1}(\P^{n+1}))$ for $k\leq 2d$.

 From the fact that the inclusion map
$i:C_{p,d}(\P^{n})\subset \mC_{p}(\P^{n})$
in equation (\ref{eqn1}) induces surjections
$i_*: \pi_k(C_{p,d}(\P^{n}))\to \pi_k(C_{p}(\P^{n}))$ for $k\leq 2d$ (cf. \cite{Lawson1}) and
$$\pi_k(C_{p,d}(\P^{n}))\cong \pi_k(C_{p}(\P^{n}))\cong
\left\{
\begin{array}{lll}
\Z, & \hbox{for $0<k\leq \min\{2d,2(n-p)\}$}\\
0, &\hbox{all other $k$ such that $k<2d$,}
\end{array}
\right.$$
we obtain also the injectivity of $i_*$ for $k\leq 2d$ since a surjective homomorphism
to from $\Z$ to $\Z$ is an isomorphism.
This completes the proof of Theorem \ref{Th2}.
\end{proof}

As applications of Theorem \ref{Th2} and Lawson's Complex Suspension Theorem \cite{Lawson1},
 we get the homotopy and homology groups of $C_{p,d}(\P^n)$ up to  $2d$.
\begin{corollary} The first $2d+1$ homotopy groups of $C_{p,d}(\P^n)$ is given by the formula
$$
 \pi_k(C_{p,d}(\P^n))\cong\left\{
\begin{array}{lll}
 \Z, &\hbox{if $k\leq min\{2d,2(n-p)\}$ and even,}\\
 0, &\hbox{all other $k<2d$.}
\end{array}
\right.$$
\end{corollary}
\begin{proof}
From the proof to Theorem \ref{Th2}, we have $\pi_k(C_{p,d}(\P^{n}))\to \pi_k(\mC_{p}(\P^{n}))$ for $k\leq 2d$.
Recall the fact that
$\mC_{p}(\P^n)$ is homotopy equivalent to the product $K(\Z, 2)\times\cdots\times K(\Z,2(n-p))$ (cf. \cite{Lawson1}),
in particular, $$\pi_k(\mC_{p}(\P^n))\cong\left\{
\begin{array}{lll}
 \Z,& \hbox{ if $k\leq 2(n-p)$ and even,}\\
 0,&\hbox{otherwise.}
\end{array}
\right.$$
\end{proof}

\begin{corollary}
The first $2d+1$ homology groups of $C_{p,d}(\P^n)$ is given by the formula
$$H_k(C_{p,d}(\P^{n}))\cong  H_k(K(\Z, 2)\times\cdots\times K(\Z,2(n-p)))$$ for $0\leq k\leq 2d$
where the right hand side can be computed by using the K\"{u}nneth formula.
\end{corollary}
\begin{proof} It also follows from the proof to Theorem \ref{Th2}
and the fact that $\mC_{p}(\P^n)$ is homotopy equivalent to
the product $K(\Z, 2)\times\cdots\times K(\Z,2(n-p))$.
\end{proof}

\section{Etale homotopy for Chow varieties over algebraically closed fields}
In this section, we will compute the Etale homotopy groups
of Chow varieties
over algebraically closed fields. Let $\P^n_K$ be the projective
space of dimension
$n$ over $K$, where $K$ is an algebraic closed field of
characteristic $char(K)$.
Let $l$ be a prime number which is different from
 $char(K)$. Let $C_{p,d}(\P^n)_K$ be
the space of effective $p$-cycles of degree $d$ in $\P^n_K$.

The notations we use in this section can be found in \cite{Friedlander1}.
Recall that the etale topological type functor $(-)_{et}$ is a functor from
simplicial schemes to pro-simplicial sets; the Bousfield-Kan
$(\Z/l)$-completion functor $(\Z/l)_{\infty}$ maps simplicial
sets to simplicial sets; the Bousfield-Kan homotopy inverse limit
functor $holim(-)$ maps indexed families of simplicial sets to
 simplicial sets;
and the geometric realization functor $Re(-)$ maps simplicial
sets to topological spaces.

\begin{definition}
Let $|(-)_{et}|:(\hbox{algebraic sets})\to (\hbox{topological spaces})$
be the functor as the composition
$Re(-)\circ holim(-)\circ(\Z/l)_{\infty}\circ(-)_{et}$.
The $k$-th etale homotopy group of $X_K$
is defined to be $\pi_k(|(X_K)_{et}|)$.
\end{definition}

Let $l_0\subset \P^n_K$ be a fixed $p$-dimensional linear subspace.
For each $d\geq 1$, we consider the closed immersions
$$
\tilde{i}:C_{p,d}(\P^n_K)\hookrightarrow C_{p,d+1}(\P^n_K).
$$
defined by $c\mapsto c+l_0$.  These immersions induce topological embeddings
\begin{equation}\label{eqn4.1}
\tilde{i}:|(C_{p,d}(\P^n)_K)_{et}|\hookrightarrow |(C_{p,d+1}(\P^n)_K)_{et}|
\end{equation}
(cf. \cite{Friedlander1}, Prop. 2.1).

From this sequence of embeddings we can form the union
$$
|\mC_{p}(\P^n_K)_{et}|:=\lim_{d\to\infty} |(C_{p,d}(\P^n)_K)_{et}|.
$$

The topology on $|\mC_{p}(\P^n_K)_{et}|$ is the weak topology for
$\{ |(C_{p,d}(\P^n)_K)_{et}|\}_{d=1}^{\infty}$. For more general discuss
on etale homotopy on spaces of algebraic cycles, the reader
is referred to the paper \cite{Friedlander1}.

Our second main result is  the following theorem.

\begin{theorem}\label{Th4.1}
For all $n,p$ and $d$, the  inclusion
$\tilde{i}:|(C_{p,d}(\P^n)_K)_{et}|\hookrightarrow |(\mC_{p}(\P^n)_K)_{et}|$
induced by Equation (\ref{eqn4.1})
is $2d$-connected.
\end{theorem}

\begin{lemma}\label{lemma4.2}
 For $k\leq 2d$, we have
$$\pi_k(|\sp^d(\P^n_K)_{et}|)=\left\{
\begin{array}{lll}
\Z_l,&\hbox{if $k\leq 2n$ and even,}\\
0, & \hbox{if $k=0$ or $k\geq 2n$ or odd.}
\end{array}
\right.
$$
\end{lemma}
\begin{proof}
We need to show that $\pi_k(|\sp^d(\P^n_K)_{et}|)\cong H_k(\P^n_K,\Z_l)$
for $0<k\leq 2d$, where $H_k(X,\Z_l)$
is the $l$-adic homology group of $X$.
First we have $|\sp^d(\P^n_K)_{et}|$ is simply
connected for any integer
$d\geq 1$. To see this, note that $|\sp^d(\P^n_K)_{et}|$ is
homotopy equivalent to $|\sp^d((\P^n_K)_{et})|$
(cf. the proof of Theorem 4.3 in \cite{Friedlander1}) and the
latter is simply connected since $|(\P^n_K)_{et}|$ is.
Since the inclusion map
 $\tilde{i}:|\sp^d(\P^n_K)_{et}|\to|\sp^{\infty}(\P^n_K)_{et}|$
is homologically $2d$-connected (cf. \cite{Milgram})
and $\pi_1(|\sp^d(\P^n_K)_{et}|)=0$,
we obtain the $2d$-connectivity of the inclusion map $i$. That is,
$\pi_k(|\sp^d(\P^n_K)_{et}|)\cong \pi_k(|\sp^{\infty}(\P^n_K)_{et}|).$
Now the theorem follows from the fact
(cf. \cite{Friedlander1}, Corollary 4.4) that
$$\pi_k(|\sp^{\infty}(\P^n_K)_{et}|)=\left\{
\begin{array}{lll}
\Z_l,&\hbox{if $0<k\leq 2n$ and even,}\\
0, & \hbox{otherwise.}
\end{array}
\right.
$$
\end{proof}

\begin{lemma}\label{lemma4.3}
For any integer $e\geq 1$, the  map
$$\tilde{\phi}_e: C_{0,d}(\P^{n})_K\to C_{0,de}(\P^{n})_K, ~ \tilde{\phi}_e(c)=e\cdot c$$
 induces injections
$\tilde{\phi}_{e*}: \pi_k(|(C_{0,d}(\P^{n})_K)_{et}|)\to \pi_k(|C_{0,de}((\P^{n})_K)_{et}|)$ for $k\leq 2d$.
\end{lemma}
\begin{proof}
Note that there is a bi-continuous algebraic morphism from $C_{0,d}(\P^n)_K$
 to $\sp^d(\P^n_K)$(cf. \cite{Friedlander1}) , we need to show
$\tilde{\phi}_{e*}:\pi_k(|(\sp^{d}(\P^{n})_K)_{et}|)\to\pi_k(|\sp^{de}((\P^{n})_K)_{et}|)$ is for $k\leq 2d$. Now the proof is word for word from Proposition \ref{lemma1.4},
except that the homotopy groups(resp. singular homology groups) are
replaced by the etale homotopy groups (resp. $l$-adic homology groups), $\Z$ (resp. $\Q$) are replaced by $\Z_l$ (resp. $\Q_l$) and the Dold-Thom theorem is replaced by the $l$-adic analogous version proved by Friedlander \cite{Friedlander1}.
\end{proof}

\begin{lemma}\label{lemma4.4}
For integers $p,d,n$, there is an integer $\tilde{e}_{p,d,n}\geq 1$ such that if $e\geq \tilde{e}_{p,d,n}$,
then the map $\widetilde{\Phi}_{p+1,d,n+1,e}:C_{p+1,d}(\P^{n})_K\to C_{p+1,de}(\P^{n+1})_K$ given by
$c\mapsto e\cdot c$ induces injections
$$(\widetilde{\Phi}_{p+1,d,n+1,e})_*: \pi_k(|(C_{p+1,d}(\P^{n+1})_K)_{et}|)\to \pi_k(|(C_{p+1,de}(\P^{n+1})_K)_{et}|)$$ for $k\leq 2d$.
\end{lemma}
\begin{proof}
The case that $K=\C$ has been proved in Proposition \ref{prop1.7}. The analogous argument is given below.

Note that the map $\widetilde{\Phi}_{p+1,d,n+1,e}:C_{p+1,d}(\P^{n})_K\to C_{p+1,de}(\P^{n+1})_K$ induced a continuous map (also denote by $\widetilde{\Phi}_{p+1,d,n+1,e}$)
$\widetilde{\Phi}_{p+1,d,n+1,e}:|(C_{p+1,d}(\P^{n+1})_K)_{et}|\to |(C_{p+1,de}(\P^{n+1})_K)_{et}|$ between topological spaces (cf. \cite{Friedlander1}, Prop. 2.1).

We also prove it by induction. The case that $p=-1$ follows from Lemma \ref{lemma4.3}. We assume that
$\Phi_{p,d,n,e}: C_{p,d}(\P^{n})_K\to C_{p,de}(\P^{n+1})_K$ defined by $\Phi_{p,d,n,e}(c)=e\cdot c$ induces injections
$(\Phi_{p,d,n,e})_*: \pi_k(|(C_{p,d}(\P^{n})_K)_{et}|)\to \pi_k(|(C_{p,de}(\P^{n})_K)_{et}|)$ for $k\leq 2d$ and $e\geq \tilde{e}_{p,d,n}$.

Let $\alpha\in \pi_k(|(C_{p+1,d}(\P^{n+1})_K)_{et}|)$ be an element such that  $(\Phi_{p+1,d,n+1,e})_*(\alpha)=0$, that is, $(F_{0D})_*(\alpha)=0$.
Let $f:S^k\to \pi_k(|(C_{p+,d}(\P^{n+1})_K)_{et}|)$ be piecewise linear up to homotopy such that $[f]=\alpha$.
By assumption, $[F_{0D}\circ f]=0$.

By Proposition 3.5 in \cite{Friedlander1} and the algebraic version of Equation (\ref{eq2}), we have
$$
\xymatrix{S^k\ar[dd]^{||}\ar[r]^-{g}&|(C_{p,d}(\P^{n})_K)_{et}|\ar[r]^{\Phi_{p,d,n,e}}\ar[d]^{\Sigma}& |(C_{p,de}(\P^{n})_K)_{et}|\ar[d]^{\Sigma}\\
&|(T_{p+1,d}(\P^{n+1})_K)_{et}|\ar[r]\ar@{^(->}@{^(->}[d]& |(T_{p+1,de}(\P^{n+1})_K)_{et}|\ar@{^(->}[d]\\
S^k\ar[r]^-{f}&|(C_{p+1,d}(\P^{n+1})_K)_{et}|\ar[r]^-{F_{0D}}\ar[ur]^{F_D}& |(C_{p+1,de}(\P^{n+1})_K)_{et}|.
}
$$
Since $F_{0D}$ is homotopy to $F_{D}:|(C_{p+1,d}(\P^{n+1})_K)_{et}|\to |(T_{p+1,de}(\P^{n+1})_K)_{et}|$
for $e\geq \tilde{e}_{p+1,d,n}$ (cf. \cite{Friedlander1}), we have
$[F_D\circ f]=0\in \pi_k(|(T_{p+1,de}(\P^{n+1})_K)_{et}|)$.
By the algebraic version of Proposition \ref{prop1.2} (cf. \cite{Friedlander1}, Prop. 3.2), $\Sigma^{-1}_*([F_D\circ f])=0$.
From the above commutative diagram and the injectivity of   $(\Phi_{p,d,n,e})_*$,
the map $f:S^k\to |(C_{p+1,d}(\P^{n+1})_K)_{et}|$ can be lifted to
a null homotopy map $g:S^k\to |(C_{p,d}(\P^{n})_K)_{et}|$
such that $(\Phi_{p,d,n,e})_*([g])=\Sigma^{-1}_*([F_D\circ f])$
and $[f]=[\Sigma\circ g]$. Hence $\alpha=[f]=0$. That is,
$(F_{0D})_*=(\Phi_{p+1,d,n+1,e})_*$ is injective  for $k\leq 2d$.
\end{proof}

\begin{proof}[The proof of Theorem \ref{Th4.1}]
 The case that $p=-1$ has been proved in \cite{Dold}, \cite{Milgram} and \cite{Friedlander1}.
By taking limit $e\to \infty$ in Lemma \ref{lemma4.4},
we get injections $\pi_k(|(C_{p+1,d}(\P^{n+1})_K)_{et}|)\to |(\mC_{p+1}(\P^{n+1})_K)_{et}|$ for $k\leq 2d$.
 On one hand, from the fact that $\pi_k(|(\mC_{p+1}(\P^{n+1})_K)_{et}|)$ is isomorphic to
 either $\Z_l$ or $0$ (cf. \cite{Friedlander1})  and the injections above,
 we obtain $\pi_k(|(\mC_{p+1}(\P^{n+1})_K)_{et}|)$ is isomorphic to either  $\Z_l$ or $0$ for $k\leq 2d$.
 On the other hand, the induced map $\tilde{i}_*:\pi_k(|(C_{p+1,d}(\P^{n+1})_K)_{et}|)\to \pi_k(|(\mC_{p+1}(\P^{n+1})_K)_{et}|)$
 by equation (\ref{eqn4.1}) is surjective for $k\leq 2d$.
 To see this, consider the commutative diagram
\begin{equation}\label{eqn4.2}
 \xymatrix{|(C_{0,d}(\P^{n})_K)_{et}|\ar[r]^{\Sigma^p}\ar[d]&|(C_{p,d}(\P^{n+p})_K)_{et}|\ar[d]^{\tilde{i}}\\
 |(\mC_{0}(\P^{n})_K)_{et}|\ar[r]^{\Sigma^p}&|(\mC_{p}(\P^{n+p})_K)_{et}|
 }
\end{equation}
 induced by
 $$
\xymatrix{C_{0,d}(\P^{n})_K\ar[r]^{\Sigma^p}\ar[d]&C_{p,d}(\P^{n+p})_K \ar[d]^{\tilde{i}}\\
 \mC_{0}(\P^{n})_K\ar[r]^{\Sigma^p}& \mC_{p}(\P^{n+p})_K
 }
 $$
 where, $\Sigma^p:=\Sigma\circ\Sigma \circ\cdots\circ\Sigma$ is the suspension for $p$ times.
 By the argument in the proof of Lemma \ref{lemma4.2}, we know that the
 left vertical arrow in equation (\ref{eqn4.2}) is a $2d$-connected
 mapping, and by Theorem 4.2 in \cite{Friedlander1} the lower
 horizontal arrow is a homotopy equivalence. This implies the map
 induced by $i$ are surjective on homotopy groups for $k\leq 2d$.
Hence  $\pi_k(|(C_{p+1,d}(\P^{n+1})_K)_{et}|)$  is isomorphic to $\pi_k(|(\mC_{p+1}(\P^{n+1})_K)_{et}|)$ for $k\leq 2d$.

 By this fact that the inclusion map
$\tilde{i}:|(C_{p,d}(\P^{n})_K)_{et}|\subset |(\mC_{p}(\P^{n})_K)_{et}|$
in equation (\ref{eqn4.1}) induces surjections
$\tilde{i}_*: \pi_k(|(C_{p,d}(\P^{n})_K)_{et}|)\to \pi_k(|(\mC_{p}(\P^{n})_K)_{et}|)$ for $k\leq 2d$ and
$$\pi_k(|(C_{p,d}(\P^{n})_K)_{et}|)\cong \pi_k(|(\mC_{p}(\P^{n})_K)_{et}|)\cong
\left\{
\begin{array}{lll}
\Z_l, & \hbox{for $0<k\leq \min\{2d,2(n-p)\}$}\\
0, &\hbox{all other $k<2d$,}
\end{array}
\right.$$
we obtain also the injectivity of $\tilde{i}_*$ for $k\leq 2d$ since a surjective homomorphism
to from $\Z_l$ to $\Z_l$ is an isomorphism.
This completes the proof of Theorem \ref{Th4.1}.
\end{proof}

As an application of Theorem \ref{Th4.1} and the Algebraic Suspension Theorem \cite{Friedlander1},
 we get the homotopy  groups of $|(C_{p,d}(\P^{n})_K)_{et}|$ up to  $2d$.
\begin{corollary} The first $2d+1$ etale homotopy groups of $C_{p,d}(\P^n)_K$ are given by the formula
$$
 \pi_k(|(C_{p,d}(\P^{n})_K)_{et}|)\cong\left\{
\begin{array}{lll}
 \Z_l, &\hbox{if $k\leq min\{2d,2(n-p)\}$ and even,}\\
 0, &\hbox{all other $k<2d$.}
\end{array}
\right.$$
\end{corollary}
\begin{proof}
From the proof to Theorem \ref{Th4.1}, we have $\pi_k(|(C_{p,d}(\P^{n})_K)_{et}|)\to \pi_k(|(\mC_{p}(\P^{n})_K)_{et}|)$ for $k\leq 2d$.
Recall the fact that
$|(\mC_{p}(\P^{n})_K)_{et}|$ is homotopy equivalent to the product $K(\Z_l, 2)\times\cdots\times K(\Z_l,2(n-p))$ (cf. \cite{Friedlander1}),
in particular, $$\pi_k(|(\mC_{p}(\P^{n})_K)_{et}|)\cong\left\{
\begin{array}{lll}
 \Z_l,& \hbox{ if $k\leq 2(n-p)$ and even,}\\
 0,&\hbox{otherwise.}
\end{array}
\right.$$
\end{proof}


\begin{thebibliography}{AAAA}

\bibitem[D]{Dold}
A. Dold,
{\sl Homology of symmetric products and other functors of complexes.}
Ann. of Math. (2) 68 1958 54--80.

\bibitem[DT]{Dold-Thom}
A. Dold and R. Thom,
{\sl Quasifaserungen und unendliche symmetrische Produkte.} (German)
Ann. of Math. (2) 67 1958 239--281.

\bibitem[ES]{Elizondo-Srinivas}
E. Javier Elizondo and V. Srinivas,
{\sl Some remarks on Chow varieties and Euler-Chow series.}
J. Pure Appl. Algebra 166 (2002), no. 1-2, 67--81. 

\bibitem[F]{Friedlander1} E. Friedlander, {\sl Algebraic cycles, Chow
varieties, and Lawson homology.}  Compositio Math. 77 (1991), no. 1,
55--93.















\bibitem[L1]{Lawson1}
H. B. Lawson, {\sl Algebraic cycles and homotopy theory.}, Ann. of
Math. {\bf 129}(1989), 253-291.

\bibitem[L2]{Lawson2}H. B. Lawson, {\sl Spaces of algebraic
cycles.} pp. 137-213 in Surveys in Differential Geometry, 1995
vol.2, International Press, 1995.






\bibitem[Mi]{Milgram}
R.J. Milgram,
{\sl The homology of symmetric products.}
Trans. Amer. Math. Soc. 138 1969 251--265.


\bibitem[W]{Whitehead} George W. Whitehead,
{\sl Elements of homotopy theory.}
Graduate Texts in Mathematics, 61. Springer-Verlag, New York-Berlin, 1978.
xxi+744 pp. ISBN: 0-387-90336-4

\end{thebibliography}
\end{document}